\documentclass[12pt]{amsart}
\usepackage{amssymb}

\title[A new  example of a $(2,5)$-distribution with large symmetry]{A new example of a generic 2-distribution on a 5-manifold with large symmetry algebra}
\author{Boris Doubrov \and Artem Govorov}

\address{Belarusian State  University,  Nezavisimosti ave.~4,  220030 Minsk, Belarus}
\email{doubrov@islc.org,  artgovorov@mail.ru}
\subjclass[2010]{58D19, 58A30, 17B25.}
\keywords{Non-holonomic vector distributions, symmetry algebras, transitive actions, absolute parallelisms}

\newcommand{\sll}{\mathfrak{sl}}
\newcommand{\so}{\mathfrak{so}}
\newcommand{\R}{\mathbb{R}}
\newcommand{\C}{\mathbb{C}}
\newcommand{\G}{\mathcal{G}}
\newcommand{\g}{\mathfrak{g}}
\newcommand{\n}{\mathfrak{n}}
\newcommand{\dd}[1]{\frac{\partial}{\partial #1}}
\newcommand{\cF}{\mathcal{F}}
\newcommand{\rtt}{\rightthreetimes}

\newtheorem{lem}{Lemma}

\theoremstyle{remark}
\newtheorem{rem}{Remark}

\begin{document}
\begin{abstract}
We discover a new example of a generic rank 2-distri\-bution on a 5-manifold with a 6-dimensional transitive symmetry algebra, which is not present  in Cartan's classical five variables paper. It corresponds to the Monge equation $z'  = y + (y'')^{1/3}$ with invariant  quartic having root type $[4]$, and a 6-dimensional non-solvable symmetry algebra isomorphic to the semidirect product of $\sll(2)$ and the 3-dimensional Heisenberg algebra.
\end{abstract}

\maketitle

\section{A new example}
The classical work of \'Elie Cartan~\cite{Car10} on the geometry of rank 2 vector distributions on 
5-dimensional manifolds is one of the most captivating papers of \'E.~Cartan's heritage and 
the source of many other constructions in modern differential geometry~\cite{CS09, GW12, Kru12a, Kru12b, LN12, Nur05, Sag06, Wil13}.

The major part of this paper is devoted to the geometry of generic 2-distributions on 5-dimensional smooth manifolds. A 2-distribution $D$ on 5-manifold $M$ is called generic, if for any (local) basis $X_1$, $X_2$ of $D$ the vector fields $X_1$, $X_2$, $X_3  = [X_1, X_2]$, $X_4  = [X_1, X_3]$, $X_5  = [X_2, X_3]$ form a (local) frame on $M$.

It is convenient to encode generic 2-distributions by so-called Monge equations,  which are the underdetermined ODE's of the form $z'  = F (x, y, y', y'', z)$.  The corresponding vector distribution is defined on the mixed order jet space $J^{2,0}(\R, \R^2) = J^2(\R, \R) \times \R$ with the standard
coordinate system $(x, y, y_1, y_2, z)$ by two vector fields:
\[
X_1 = \dd{y}, \quad X_2 = \dd{x} + y_1 \dd{y} + y_2 \dd{y_1} + F\dd{z}.
\]
It is generic if and only if  $\frac{\partial^2F}{(\partial y_2)^2}  \ne 0$. In fact, any generic 2-distribution can be encoded by a certain Monge equation~\cite{Gou22} (see also~\cite{Kru12a,Str09}).

Given a generic 2-distribution $D$ on a 5-dimensional manifold $M$, \'Elie Cartan constructs a natural 14-dimensional principal bundle $\pi\colon \G\to M$ and a $\g_2$-valued absolute parallelism on $\G$, where
$\g_2$ is the split real form of the exceptional simple Lie algebra of type $G_2$.  In particular,
this proves that the symmetry algebra of any generic 2-distribution is finite-dimensional and of dimension $\le 14$.

Cartan also constructs the first non-trivial  local invariant,  which happens to be a quartic $\cF \in S^4(D^*)$. He proves that a given generic 2-distribution has exactly 14-dimensional symmetry algebra if and only if $\cF$ vanishes identically.  In this case, it is locally equivalent  to the distribution defined by the Monge 
equation  $z' = (y'')^2$.

He then proceeds to analyze the cases when $\cF$ does not vanish identically. His main motivation is to find 
all cases when there is no frame on $M$ itself naturally associated to the distribution $D$~\cite[p.162]{Car10}: 
\begin{quotation}
Nous ne voulons pas entrer dans la discussion compl\`ete de tous les cas qui peuvent se pr\'esenter 
comme interm\'e\-diaires entre le cas g\'en\'eral et le cas particulier o\`u la forme $\cF$ est identiquement nulle.  Nous allons simplement rechercher tous les cas dans lesquels il est impossible de trouver 
cinq expressions de Pfaff covariantes $\omega_1$, $\omega_1$, \dots, $\omega_5$, ou, autrement dit, tous les cas dans lesquels il est impossible de former cinq param\`etres diff\'erentiels lin\'eaires ind\'ependants.  
Il est n\'ecessaire pour cela que la forme $\cF$ contienne ou, un facteur lin\'eaire triple ou deux
facteurs lin\'eaires doubles, ou un facteur lin\'eaire quadruple. C'est par ce dernier cas que nous commencerons.
\end{quotation}

Clearly, such a frame on the manifold $M$ does not exist, if the symmetry algebra of the distribution 
$D$ has non-trivial stabilizer on the manifold $M$. Cartan shows that such cases may happen only if the
quartic $\cF$ has either one root of multiplicity $4$ (root type $[4]$) or two different roots with 
multiplicity $2$ (root type $[2, 2]$).

In case of root type $[4]$ he states that there is exactly one family of such distributions parametrized by an additional local invariant $I$ (see also~\cite[Section 17.4]{Sto00} for a translation of \'Elie Cartan's computations to English). When $I$ is constant\footnote{It is a general belief (see, for example,~\cite{Kru12b,Wil13}) that these systems with constant~$I$ can also be encoded by 
the simpler Monge equation $z'  = (y'' )^m$ with $m \ne -1, 1/3, 2/3, 2$. However, this family 
misses a single case of $I = \pm 3/4$, which is equivalent, for example, to $z'  = \ln(y'')$.}, 
the symmetry algebra is 7-dimensional, and the distribution itself is encoded by the following Monge equation (see~\cite[p.171,eqn.6]{Car10})\footnote{Cartan's paper has here the coefficient $\frac{5}{6}$ in place 
of $\frac{10}{3}$. But this is an arithmetic error corrected by Strazzullo~\cite{Str09}.}:
\begin{equation}\label{7dim}
z' = -1/2 \left( (y'')^2 + \tfrac{10}{3}(y')^2 + (1+I^2)y^2\right).
\end{equation}
If $I$ is non-constant then the symmetry algebra is 6-dimensional non-transitive.

Note that the above description of all generic 2-distributions with exactly 7-dimensional symmetry algebra was also proven by other authors. For example, papers~\cite{Kru12a,KT13} reprove that $7$ is the maximal 
possible dimension for the symmetry algebra for non-flat case (i.e., for the case when $\cF$ does not vanish identically). The paper~\cite[Th.2]{DZ13} reproves that all such 2-distributions are described by the Monge equation:
\[
z' = (y'')^2 + r_1(y')^2 + r_2y^2,
\]
where the pair $(r_1, r_2)$ is viewed up to the transformations $(r_1, r_2)\mapsto (c^2r_1, c^4r_2)$, $c \ne 0$, and the roots of the equation $\lambda^4 - r_1\lambda^2 + r_2 = 0$ do not constitute an arithmetic progression. If the latter condition is not satisfied, then the symmetry algebra becomes 14-dimensional.

In case of root type $[2, 2]$, Cartan finds a number of cases with exactly 6-dimensional symmetry algebra,
which is isomorphic (after complexification) to one of the following:  $\sll(2, \C) \times \sll(2, \C)$, 
$\so(3, \C) \rtt \C^3$, $\sll(2, \C) \times (\so(2, \C) \rtt \C^2)$.

\textbf{
We show that the analysis of root type $[4]$ case is incomplete and misses the generic 
2-distribution encoded by the Monge system:
\begin{equation}\label{newex}
z' = y + (y'')^{1/3}
\end{equation}
}
This equation is a part of larger family $z' = y + (y'')^m$, which has only 5-dimensional symmetry algebra for generic $m$ (see the forthcoming paper~\cite{DG13}).

\begin{lem}\label{lem1}
The 2-distribution corresponding to~\eqref{newex} has invariant quartic with root type $[4]$.
\end{lem}
\begin{proof}
We use the formulas from~\cite{GW12,Hsi80,Nur05} to compute the invariant quartic. The exact formulas are tedious and are omitted here.
\end{proof}

\begin{lem}\label{lem2}
The symmetry algebra of~\eqref{newex} is given by the following 6 vector fields:
\begin{align*}
S_1 &= x\dd{y}+\dd{y_1}+\frac{1}{2}x^2\dd{z},\\
S_2 &= x\dd{x} - y\dd{y} -2y_1\dd{y_1} -3y_2\dd{y_2},\\
S_3 &= y\dd{x} - y_1^2\dd{y_1} - 3 y_1y_2\dd{y_2} + \frac{1}{2}y^2\dd{z},\\
S_4 &= \dd{x},\\
S_5 &= \dd{y} + x\dd{z},\\
S_6 &= \dd{z}.
\end{align*}
It is isomorphic to $\sll(2)\rtt\n_3$ and preserves the fibration 
$\pi\colon J^{2,0}(\R, \R) \to J^2(\R, \R)$. The restriction of $\pi_*$
 to this symmetry algebra has a 1-dimensional kernel spanned by the 
center $\langle S_6\rangle$ and a 5-dimensional image that coincides 
with the prolongation of the standard action of the equiaffine Lie algebra 
on the plane.
\end{lem}
\begin{proof} Simple computation shows that indeed all vector fields $S_1$, \dots, $S_6$
are symmetries of the distribution.  As all these symmetries do not depend on $z$, we see 
that the fibration $\{z = \operatorname{const}\}$ is indeed preserved by the symmetry algebra 
and that $S_6$ lies in its center. Projecting these vector fields to $J^2(\R, \R)$ we get 5 vector 
fields ($S_6$ projects to $0$), which preserve the contact distribution on $J^2(\R,\R)$. It is also 
easy to see that they are prolongations of 5 generators for the equiaffine action on the plane:
\[
x\dd{y},\ x\dd{x}-y\dd{y},\ y\dd{x},\ \dd{x},\ \dd{y}.
\]
We note that $\langle S_1, S_2, S_3\rangle$ is a subalgebra isomorphic to $\sll(2)$, while
$\langle S_4, S_5, S_6\rangle$ is a radical isomorphic to the 3-dimensional 
Heisenberg algebra.

Let us now show that this is indeed the full symmetry algebra of the given distribution. 
Otherwise it would be a subalgebra of a certain bigger symmetry algebra. As this distribution 
has root type $[4]$ this bigger symmetry algebra could only be one of the 7-dimensional 
symmetry algebras of~\eqref{7dim}. But these are all solvable and cannot contain a 
semisimple subalgebra isomorphic to $\sll(2)$.                                       
\end{proof}

\begin{rem}
Apparently, the distribution corresponding to~\eqref{newex} was first discovered by 
Francesco Strazzullo in his PhD thesis~\cite{Str09} as  Example~6.7.2. It is given there 
as a Monge equation
\[
z' = 1 + \exp(-4y/3)\big(y'' - 1/2(y')^2\big)^{2/3}.
\]
However, as he used Maple for computing symmetry algebras, it was not clear whether the 6-dimensional symmetry algebra he computed, is the full symmetry algebra. Although the same argument as above could be used in his case as well.
\end{rem}

\begin{rem}
To be independent of Lemma~\ref{lem1} and related calculations, let us show that 
equation~\eqref{newex} has non-vanishing  invariant quartic. Indeed,  otherwise its symmetry 
algebra would be embedded into $\g_2$. But this is not possible,  as up to 
conjugation $\g_2$ has only one non-zero element that has a centralizer of 
dimension $\ge 6$ (see~\cite{Elk72}). It is a nilpotent element that corresponds to 
the longest root of $G_2$.  Its centralizer is 8-dimensional and is isomorphic to 
the semidirect product of $\sll(2)$ and a 5-dimensional Heisenberg algebra. But there are 
no Lie algebra injective  homomorphisms $\sll(2) \rtt \n_3  \to \sll(2) \rtt \n_5$  which 
map $\sll(2)$ to $\sll(2)$ (we can assume this as all Levi subalgebras are conjugate 
to each other) and the center to the center.
\end{rem}

\begin{rem} We shall show in the forthcoming paper~\cite{DG13} that this is the only 
missing case of a generic 2-distribution with transitive symmetry algebra of dimension $\ge 6$. 
The proof is based on the observation that over complex numbers any such distribution 
admits a simply transitive 5-dimensional subalgebra of the full symmetry algebra.
\end{rem}

\section*{Acknowledgements} We would like to thank Dennis The, Trawis Willse, Boris Kruglikov for their valuable comments and independent verification of the statements in this note. We are also grateful to Ian Anderson for providing us the Maple code based on \textsc{DifferentialGeometry} package that 
computes Cartan's frame for any given Monge equation.


\begin{thebibliography}{99}
\bibitem{Car10} \'Elie Cartan, \emph{Les syst\'emes de Pfaff \'a cinq variables et les \'equations aux d\'eriv\'ees partielles du second ordre},  Annales scientifiques de l'\'E.N.S. 3e s\'erie, \textbf{27}, 1910,
109--192.

\bibitem{CS09} Andreas \v Cap and Jan Slov\'ak,  \emph{Parabolic Geometries I}, Mathematical Surveys and Monographs. American Mathematical Society, Providence, 2009.

\bibitem{DG13} Boris Doubrov, Artem Govorov, \emph{Classification of generic 2-distributions on
5-manifolds with simply transitive symmetry algebras}, in preparation.
 
\bibitem{DZ13} Boris Doubrov, Igor Zelenko, \emph{Geometry of rank 2 distributions with nonzero Wilczynski invariants and affine control systems with one input}, Arxiv preprint arXiv:1301.2797,  pages 1--27, 2013.

\bibitem{Elk72} Gordon B.~Elkington, \emph{Centralizers of Unipotent Elements in Semisimple Algebraic Groups}, Journal of Algebra \textbf{23}, 1972, 137--163.

\bibitem{Gou22} \'Edouard Goursat, \emph{Le\c cons  sur le probl\`eme  de Pfaff},  Librairie Scientifique J.
Hermann, Paris, 1922.

\bibitem{GW12} C.~Robin Graham and Travis Willse, \emph{Parallel tractor extension and ambient metrics of holonomy split $G_2$}, J.~Diff.~Geom., \textbf{92}, 2012, 463--505.

\bibitem{Hsi80} S.H.C.~Hsiao, \emph{On Cartan-Sternberg's  example of the reduction of a $G$-structure
II-III}, Tamkang J.~Math., \textbf{11}, 1980, 277--312.

\bibitem{Kru12a} Boris Kruglikov, \emph{Lie theorem via rank 2 distributions (integration of PDE of 
class  $\omega = 1$)}, J.~Nonlinear Math.~Phys. \textbf{19}, 2012, 1250011.

\bibitem{Kru12b} Boris Kruglikov, \emph{The gap phenomenon in the dimension study of finite type systems}, Central European J.~Math., \textbf{10}, 2012, 1605--1618.

\bibitem{KT13} Boris Kruglikov,  Dennis The, \emph{The gap phenomenon in parabolic geometries}, Arxiv preprint arXiv:1303.1307,  pages 1--54, 2013.

\bibitem{LN12} Thomas Leistner and  Pawe\l{} Nurowski,  \emph{Conformal  structures  with  $G_{2(2)}$-ambient metrics}, Ann.~Sc.~Norm.~Super.~Pisa Cl.~Sci., \textbf{11}, 2012, 407--436.

\bibitem{Nur05} Pawe\l{} Nurowski, \emph{Differential equations and conformal structures}, 
Journ.~Geom.~Phys. \textbf{55}, 2005, 19--49.

\bibitem{Sag06} Katja  Sagerschnig, \emph{Split octonions and generic rank two distributions in dimension five}, Arch.~Math. (Brno), \textbf{42}, 2006, 329--339.

\bibitem{Sto00} Olle Stormark, \emph{Lie's Structural Approach to PDE Systems}, Encyclopedia of
Mathematics and Its  Applications. Cambridge University  Press, Cambridge, 2000.

\bibitem{Str09} Francesco Strazzullo, \emph{Symmetry Analysis of General Rank-3 Pfaffian Systems in Five Variables}, PhD thesis, Utah State University, 2009.

\bibitem{Wil13} Travis Willse, \emph{Highly symmetric 2-plane fields on 5-manifolds and Heisenberg
5-group holonomy}, Arxiv  preprint arXiv:1302.7163,  pages 1--26, 2013.
\end{thebibliography}
\end{document}